\documentclass[11pt]{amsart}
\usepackage{graphics,graphicx}
\usepackage{xcolor}
\usepackage{amsmath}
\usepackage{amsthm,amstext,amsfonts,bm,amssymb}
\usepackage{a4wide}
\usepackage{float}
\usepackage{enumerate}

\newtheorem{theorem}{Theorem}[section]

\newtheorem{proposition}[theorem]{Proposition}
\newtheorem{corollary}[theorem]{Corollary}
\newtheorem{conjecture}[theorem]{Conjecture}

\theoremstyle{definition}

\newtheorem{example}[theorem]{Example}

\theoremstyle{remark}

\begin{document}

\title[Restricted and Associated Bell and Factorial Numbers]{Combinatorial and Arithmetical  Properties of the Restricted and Associated Bell and Factorial Numbers}

\author{Victor H. Moll}
\address{Department of Mathematics,
Tulane University, New Orleans, LA 70118}
\email{vhm@tulane.edu}

\author{Jos\'{e} L. Ramirez}
\address{Departamento de  Matem\'{a}ticas,
Universidad Nacional de Colombia, Bogot\'{a}, Colombia}
\email{jlramirezr@unal.edu.co}

\author{Diego Villamizar}
\address{Department of Mathematics,
Tulane University, New Orleans, LA 70118}
\email{dvillami@tulane.edu}


\date{\today}


\begin{abstract}
Set partitions and permutations with restrictions on the size of the blocks and cycles are important combinatorial sequences.
Counting these objects lead to the sequences  generalizing the classical Stirling and Bell numbers.  The main focus of the
present article is the analysis of combinatorial and arithmetical properties of them. The results include several combinatorial identities and recurrences as well as some properties of their $p$-adic valuations.

\end{abstract}

\maketitle

\newcommand{\N}{{\mathbb N}}
\newcommand{\realpart}{\mathop{\rm Re}\nolimits}
\newcommand{\imagpart}{\mathop{\rm Im}\nolimits}

\numberwithin{equation}{section}

\section{Introduction} \label{intro}
\setcounter{equation}{0}

The (unsigned) Stirling numbers of the first kind  denoted by $c(n,k)$ or ${n \brack k}$  enumerate the number of permutations
on  $n$ elements with $k$ cycles.  The corresponding Stirling numbers of the second kind,  denoted
by  $S(n,k)$ or ${n \brace k}$, enumerate the number of partitions of a set with $n$ elements into $k$ non-empty blocks; see
\cite{comtet-1974a} for general information about them.  The recurrences
\begin{align*}
{n+1 \brack k}&={n \brack k-1}+n{n \brack k} \quad \text{and}\\
{n+1 \brace k}&={n \brace k-1}+k{n \brace k},
\end{align*}
with the initial conditions
 \begin{align*}
 {0 \brack 0}&=1, \quad {n \brack 0}={0 \brack n}=0,\\
 {0 \brace 0}&=1, \quad {n \brace 0}={0 \brace n}=0,
 \end{align*}
hold for $n\geq 1$. They are related to each other by the orthogonality relation
$$\sum_{k\geq 0} {n \brace k}{k \brack m}(-1)^{n-k} =\delta_{n,m},$$
 where $\delta_{n,m}$ is the Kronecker delta function.



The Bell numbers, $B_n$, enumerate the set partitions of a set with $n$ elements, so that
$\begin{displaystyle} B_n=\sum_{k=0}^n{n \brace k} \end{displaystyle}$. The \textit{Spivey's formula} \cite{spivey-2008a}
\begin{align} \label{spiveye}
B_{n+m}=\sum_{k=0}^n\sum_{j=0}^mj^{n-k}  {m \brace j} \binom{n}{k}B_k,
\end{align}
gives a recurrence for them. Further properties of this sequence appear in \cite{comtet-1974a, mansour-2012a}.\\

The literature contains several generalizations of Stirling numbers; see \cite{mansour-2015a}. Among them, the so-called
restricted and associated Stirling numbers of both kinds (cf. \cite{bona-2016a, choijy-2005a,choijy-2006b,choijy-2006a,
comtet-1974a, komatsu-2015c, komatsu-2015b,  mezo-2014a}) constitute the central character of the work presented here.

The \textit{restricted Stirling numbers of the second kind}  $ {n \brace k}_{\leq m }$ give the number of partitions
 of $n$ elements  into $k$ subsets, with the additional restriction that none of the blocks contain more than
 $m$ elements. Komatsu et al. \cite{komatsu-2015c}  derived the recurrences
 \begin{align}\label{recur-1}
{n+1 \brace k}_{\leq m}  =\sum_{j=0}^{m-1} \binom{n}{j}
 {n-j \brace k-1}_{\leq m}   = k  {n \brace k}_{\leq m}   + {n \brace k-1}_{\leq m}   -     \binom{n}{m} {n-m \brace k-1}_{\leq m},
 \end{align}
     \noindent
with initial conditions
$ {0 \brace 0}_{\leq m}   = 1$ and  $   {n\brace 0}_{\leq m}  = 0$, for $n \geq 1$.  The \textit{restricted Bell numbers} defined by \cite{miksa-1958a}
 \begin{equation*}
 B_{n, \leq m} = \sum_{k=0}^{n}  {n \brace k}_{\leq m },
 \end{equation*}
 enumerate partitions of $n$ elements into blocks, each one of them with at most $m$ elements.  For
  example, $B_{4, \leq 3} = 14$, the partitions being
  \begin{align*}
 & \left\{ \{1 \}, \{ 2 \}, \{ 3 \}, \{ 4 \}  \right\}, & & \left\{ \{ 1, \, 2 \}, \{ 3 \}, \{ 4 \}  \right\}, & & \left\{ \{ 1, \, 2 \}, \{ 3, \, 4 \} \right\}, &  &  \left\{ \{ 1, \, 3 \}, \{ 2 \}, \{ 4 \} \right\}, & \\
 & \left\{ \{ 1, \, 3 \}, \{ 2, \, 4 \} \right\}, & & \left\{ \{ 1, \, 4 \}, \{ 2 \}, \{ 3 \} \right\}, & &  \left\{ \{ 1, \, 4 \}, \{ 2, \, 3 \} \right\}, &  & \left\{ \{ 1, \, 2, \, 3 \}, \{ 4 \} \right\}, &\\
  & \left\{ \{ 1, \, 2, \, 4 \}, \{  3 \} \right\}, & &  \left\{ \{ 1, \, 3, \, 4 \}, \{ 2  \} \right\}, & & \left\{\{ 1 \},  \{ 2, \, 3, \, 4 \}\right\}, & & \left\{ \{ 1 \},  \{ 2 \}, \, \{ 3, 4 \} \right\}, & \\
&  \left\{ \{ 1 \},  \{ 2, 4 \}, \{ 3 \} \right\}, & & \left\{ \{ 1 \},  \{ 2, 3 \}, \{ 4 \} \right\}. & &  &
  \end{align*}

An associated sequence is the \textit{restricted Stirling numbers of the first kind}  $ {n \brack k}_{\leq m }$. This gives the  number of
permutations on  $n$ elements with $k$ cycles with the restriction that none of the cycles  contain more than
$m$ items (see  \cite{mezo-2014a} for more information).  Komatsu et al. \cite{komatsu-2015b} established the recurrence
 \begin{align}
{n+1 \brack k}_{\leq m}  =\sum_{j=0}^{m-1} \frac{n!}{(n-j)!}
 {n-j \brack k-1}_{\leq m} = n {n \brack k}_{\leq m}   + {n \brack k-1}_{\leq m}   -     \frac{n!}{(n-m)!} {n-m \brack k-1}_{\leq m},
 \end{align}
     \noindent
     with initial conditions
    $ {0 \brack 0}_{\leq m}   = 1$ and  $ {n\brack 0}_{\leq m}  = 0$. The  \textit{restricted   factorial numbers}, see \cite{mezo-2014a}, are
    defined by
 \begin{equation*}
 A_{n, \leq m} = \sum_{k=0}^{n}  {n \brack k}_{\leq m }.
 \end{equation*}
 These enumerate all permutations  of $n$ elements into cycles with the condition that every cycle has at most $m$ items. For
  example, $A_{4, \leq 3} = 18$ with the permutations being
  \begin{align*}
 & (1)(2)(3)(4), && (1)(2)(43), & & (1)(32)(4), & & (1)(342), &  & (1)(432), & \\
& (1)(42)(3),  & & (21)(3)(4), & &  (21)(43), & & (231)(4), &  &  (241)(3), & \\
& (321)(4), & &  (31)(2)(4), & &  (341)(2), & & (31)(42), & & (421)(3), & \\
&  (431)(2), & & (41)(2)(3), && (41)(32).
  \end{align*}

The outline of the paper is this: Section 2 contains some known identities of the restricted Bell numbers $B_{n, \leq 2}$. In this
case, $m=2$, the restricted Bell and restricted factorial numbers coincide, i.e., $B_{n, \leq 2}=A_{n, \leq 2}$.  Information about their
Hankel transform is included. Section 3 contains extensions of these properties to $m=3$ and Sections 4 and 5  present the general
case. Section 6 establishes the log-convexity of the restricted Bell and factorial sequences, extending  classical results. Some conjectures on the roots of the restricted Bell polynomials are proposed here. Finally, Section 7 presents  some preliminary
results on the $p$-adic valuations of  these sequences.  Explicit expressions for the prime $p=2$ are established. A more complete
discussion of these issues is in preparation.

\section{Restricted Bell numbers $B_{n, \leq 2}$ and  restricted factorial numbers $A_{n, \leq 2}$}
\label{sec-restricted2}

This section discusses the sequence $B_{n, \leq 2}$, which enumerates partitions of $n$ elements into blocks of
length at most $2$. Then $B_{n, \leq 2}=A_{n, \leq 2}$ is precisely the number of \textit{involutions} of the $n$ elements, denoted in
\cite{amdeberhan-2015b} by  $\text{Inv}_{1}(n)$. This sequence is also called  \textit{Bessel numbers of the second
kind}, see \cite{choijy-2003a} for further information.

   The well-known  recurrence
 \begin{equation}\label{recBn2}
 B_{n, \leq 2} = B_{n-1,\leq 2} + (n-1) B_{n-2, \leq 2},
 \end{equation}
 \noindent
 with initial conditions $B_{0,\leq 2} = B_{1, \leq 2} = 1$, yields the exponential generating function
 \begin{equation}\label{expogen}
 \sum_{n=0}^{\infty} B_{n, \leq 2} \, \frac{x^{n} }{n!} = \exp{ \left( x + \tfrac{1}{2} x^{2} \right)}
 \end{equation}
 \noindent
 as well as the closed-form expression
 \begin{equation}\label{eqr1}
 B_{n, \leq 2} = \sum_{j=0}^{\lfloor n/2 \rfloor} \binom{n}{2j} \frac{(2j)!}{2^{j} \, j!}.
 \end{equation}
 The recurrence
 \begin{equation}
 B_{n_{1}+n_{2}, \leq 2}  =  \sum_{k \geq 0} k! \binom{n_{1}}{k} \binom{n_{2}}{k} B_{n_{1}-k, \leq 2} B_{n_{2}-k, \leq 2}
 \label{recu-1a}
 \end{equation}
 \noindent
 is established in \cite{amdeberhan-2015b}.

 Congruences for the involution numbers appeared in Mez\H{o} \cite{mezo-2014a},
 in a problem on the distribution of last digits of related sequences. These include
 \begin{equation}
B_{n, \leq 2} \equiv B_{n+5, \leq 2} \bmod 10 \text{ if } n > 1 \ \text{ and }
 B_{n, \leq 3} \equiv B_{n+5, \leq 2} \bmod 10 \text{ if  } n > 3.
 \end{equation}

\subsection{The Hankel Transform of $B_{n,\leq 2}$}

For a sequence $A =\left(a_n\right)_{n\in\N}$, its \textit{Hankel matrix} $H_n$ of order $n$  is defined  by
\begin{align*}
H_n=\begin{bmatrix}
a_0 & a_1 & a_2 & \cdots & a_n\\
a_1 & a_2 & a_3 & \cdots & a_{n+1}\\
\vdots & \vdots & \vdots &  & \vdots \\
a_n & a_{n+1} & a_{n+2} & \cdots & a_{2n}
\end{bmatrix}.
\end{align*}
The \textit{Hankel transform} of $A$ is the sequence $\left( \det H_{n} \right)_{n \in \mathbb{N}}$. Aigner
\cite{aigner-1999a} showed that the
Hankel transform of the Bell numbers is the sequence of the product of first $n$ factorials, so-called \textit{superfactorials}
, i.e., $(1!, 1!2!, 1!2!3!, \dots)$.  Theorem \ref{hankel-2} below shows that the Hankel transform of $B_{n,\leq 2}$ is
also given by superfactorials.

The first result gives the binomial transform of $B_{n, \leq 2}$. This involves the \textit{double factorials}
$$(2n-1)!!=\prod_{k=1}^n(2k-1)=\frac{(2n)!}{n!2^n}.$$


\begin{proposition}\label{propo}
The  binomial transform of the sequence $B_{n,\leq 2}$ is
$$\sum_{i=0}^n(-1)^i\binom{n}{i}B_{i,\leq 2}=\begin{cases}(n-1)!!,& \text{if} \ n \ \text{is even;} \\ 0,& \text{if} \ n  \text{\ is odd.}  \end{cases}$$
The numbers on the right are called the aerated double factorial.
\end{proposition}
\begin{proof}
The exponential generating function $A(x)$  of a sequence $(a_n)_{n\geq 0}$ and that
of its binomial transform $S(x)$ are related by $S(x)=e^{-x}A(x)$. The result now follows from (\ref{expogen}).
\end{proof}

\noindent
\emph{Combinatorial Proof of Proposition \ref{propo}:} Let $\mathcal{B}_{n,\leq 2}$ be the set of all partitions into blocks of
length at most 2. Let $\mathcal{S}_{n,i}=\{\pi \in \mathcal{B}_{n,\leq 2}:\{i\}\in \pi\}$ be the set of partitions of $[n]$ in blocks of length less or equal to $2$, where $i$ is a singleton block. There are $B_{n-1,\leq 2}$ of them.  Then $$\mathcal{B}_{n,\leq 2}=\bigcup _{i=1}^{n}\mathcal{S}_{n,i} \,\, \bigcup \underbrace{(\mathcal{B}_{n,\leq 2}\setminus (\bigcup _{i=1}^{n}\mathcal{S}_{n,i}))}_{\text{Denote this
 by $L_{n}$}}.$$
 \noindent
The inclusion-exclusion principle gives $$B_{n,\leq 2}=\sum _{i=1}^n(-1)^{i-1}\binom{n}{i}B_{n-i,\leq 2}+|L_n|,$$ that yields
$$|L_n|=B_{n,\leq 2}-\sum _{i=1}^n(-1)^{i-1}\binom{n}{i}B_{n-i,\leq 2}=\sum _{i=0}^n(-1)^{i}\binom{n}{i}B_{n-i,\leq 2}=\sum _{i=0}^n(-1)^{n-i}\binom{n}{i}B_{i,\leq 2}.$$
On the other hand, $L_n=\{\pi \in \mathcal{B}_{n,\leq 2}: \text{such that if $B\in \pi$ then $|B|=2$}\}$, because it is the complement
of the partitions with at least one singleton. Thus
 $$|L_n|=\frac{\binom{n}{2,2,\dots ,2}}{ \left(\frac{n}{2} \right)!}=\frac{n!}{2^{n/2} \left(\frac{n}{2} \right)!}$$ if $n$ is
 even and $0$ if $n$ is odd. This establishes the identity. \\

Barry and Hennessy \cite[Example 16]{barry-2010a} show that Hankel transform of the aerated double
factorial is the superfactorials. The
fact that any integer sequence has the same Hankel transform as its binomial transform \cite{layman-2001a, spivey-2006a}, gives the
next  result.

\begin{theorem}
\label{hankel-2}
The Hankel transform of the restricted Bell numbers $B_{n,\leq 2}$ is the superfactorials; that is, for  any fixed $n$,
$$\det\begin{bmatrix}
B_{0,\leq 2} & B_{1,\leq 2} & B_{2,\leq 2} & \cdots & B_{n,\leq 2}\\
B_{1,\leq 2} & B_{2,\leq 2} & B_{3,\leq 2} & \cdots & B_{n+1,\leq 2}\\
\vdots & \vdots & \vdots &  & \vdots \\
B_{n,\leq 2} & B_{n+1,\leq 2} & B_{n+2,\leq 2} & \cdots & B_{2n,\leq 2}
\end{bmatrix}=\prod_{i=0}^{n}i!.$$
\end{theorem}

\bigskip
\section{The Restricted Bell numbers $B_{n, \leq 3}$ and the restricted factorial numbers $A_{n, \leq 3}$}
 \label{sec-recurrence}

 The goal of the current section is to extend some results in the previous section to the case $m=3$. Recurrences
 established here are employed in Section \ref{sec-arith} to discuss arithmetic properties of $B_{n, \leq 3}$ and $A_{n, \leq 3}$.

 \smallskip

 The first statement relates $B_{n, \leq 3}$ to the involution numbers $B_{n, \leq 2}$.

 \begin{theorem}
 \label{teorecr1}
 The restricted Bell numbers $B_{n, \leq 3}$ are given by
 \begin{equation}
B_{n, \leq 3}  =  \sum_{j=0}^{\lfloor{ n/3 \rfloor}} \binom{n}{3j} \frac{(3j)!}{(3!)^{j} \, j!} B_{n-3j, \leq 2}.
\label{rec-1}
 \end{equation}
 \end{theorem}
\begin{proof}
Count the set of all partitions of $[n]$ into block of size at most 3, with exactly $j$ blocks  of size 3. To do so, first choose a
 subset of $[n]$ of size $3j$ to place the $j$ blocks of size 3. This is done in $\binom{n}{3j}$ ways. Then, the number of set partitions
 of $[3j]$  such that each block has three elements is  $\frac{(3j)!}{(3!)^{j} \, j!}$.  The remaining $n-3j$ elements produce
  $B_{n-2j,\leq 2}$ partitions.   Summing over $j$ completes the argument.
\end{proof}

 The next result gives a recurrence for $B_{n, \leq 3}$.

 \begin{theorem}\label{req1}
 The restricted Bell  numbers $B_{n, \leq 3 }$ satisfy the recurrence
 \begin{equation}
 B_{n, \leq 3} =  B_{n-1, \leq 3} + \binom{n-1}{1} B_{n-2, \leq 3} + \binom{n-1}{2} B_{n-3, \leq 3},
 \label{rec-stir1}
 \end{equation}
 with initial conditions  $B_{0, \leq 3} = 1, \, B_{1, \leq 3} = 1, \, B_{2, \leq 3} = 2.$
 \end{theorem}
 \begin{proof}
 The expression for $B_{n, \leq 2}$ in \eqref{eqr1} and \eqref{rec-1} produce
 \begin{equation*}
 B_{n, \leq 3} = \sum_{i=0}^{\lfloor \tfrac{n}{3} \rfloor } \sum_{j=0}^{\lfloor \tfrac{n}{2} \rfloor }
 \binom{n}{3i} \frac{(3i)!}{6^{i} i!} \binom{n-3i}{2j} \binom{2j}{j} \frac{j!}{2^{j}},
 \end{equation*}
 \noindent
 that may be written as
 \begin{equation*}
 B_{n, \leq 3} = \sum_{i=0}^{\lfloor \tfrac{n}{3} \rfloor } \sum_{j=0}^{\lfloor \tfrac{n}{2} \rfloor}
 \binom{n}{3i+2j} \binom{3i+2j}{2j} \binom{2j}{j} \frac{(3i)! \, j!}{6^{i} \, i! \, 2^{j}}.
 \end{equation*}
 \noindent
 The recurrence is obtained as a routine application of the WZ-method \cite{nemesi-1997a, petkovsek-1996a}.
 \end{proof}

 \noindent
\emph{Combinatorial proof of Theorem \ref{req1}}: Suppose the first block is the size $i$ with $i=1, 2$ or 3. Since this block
contains the minimal element, one only needs to  choose  $l$ elements, with $l=0, 1$ or 2. Therefore, the number of set
partitions of $[n]$ with exactly $i$ elements in the first block is given by $\binom{n-1}{i}B_{n-i,\leq3}$ for $i=1, 2, 3$. Summing over $i$
completes the argument.

\medskip

The above recurrence produces the exponential generating function
 \begin{equation}\label{genfun3}
  \sum_{n=0}^{\infty} B_{n, \leq 3} \, \frac{x^{n} }{n!} = \exp{ \left( x + \tfrac{1}{2} x^{2} + \tfrac{1}{3!} x^{3} \right)}.
 \end{equation}

 The next result is an extension of \eqref{recu-1a} for the case of partitions with blocks of length at most $3$. It is an analog of the
 Spivey-like formula \eqref{spiveye}.

 \begin{theorem}
 \label{teo1a}
 Define $a(i,j) = \frac{1}{3}(2i-j-k)$. Then the restricted Bell numbers $B_{n, \leq 3}$ satisfy the relation
 \begin{eqnarray*}
 B_{n+m, \leq 3} & = & \sum_{i=0}^{n} \binom{n}{i} B_{n-i, \leq 3}
 \sum_{j= \left\lceil  \tfrac{i}{2} \right\rceil}^{\min \{m, 2i \} } \binom{m}{j} B_{m-j, \leq 3}  \\
 & & \times \sum_{\substack{k=0 \\ k\equiv -i-j \bmod{3}}}^{\min \{ i, j, 2i-j,2j-i \}} \binom{i}{k} \binom{j}{k} k!
 \binom{i-k}{a(j,i)} \binom{j-k}{a(i,j)} \frac{(2a(i,j))! \, (2a(j,i))!}{2^{a(i,j) + a(j,i)}} \nonumber\\
 & = &
 \sum _{i=0}^n \sum _{j=\lceil \frac{i}{2}\rceil}^{\min \{m,2i\}} \sum_{\substack{k=0 \\ k\equiv -i-j \bmod{3}}}^{\min \{i,j,2i-j,2j-i\}}\frac{n!m!B_{n-i,\leq 3}B_{m-j, \leq 3}}{k!(n-i)!(m-j)!a(i,j)!a(j,i)!2^{\frac{i+j-2k}{3}}}.
 \end{eqnarray*}
 \end{theorem}
 \begin{proof}
 The set of $n+m$ elements whose partitions are enumerated by $B_{n+m, \leq 3}$ is split into two disjoint sets
 $I_{1}$ and $I_{2}$ of cardinality $n$ and $m$, respectively. Any such partition $ \pi$ can be written uniquely
 in the form $\pi = \pi_{1} \cup \pi_{2} \cup \pi_{3}$, where $\pi_{1}$ is a partition of a subset of $I_{1},  \, \pi_{2}$ is a
 partition of  a subset of $I_{2}$ and the blocks in $\pi_{3}$ contain elements of both $I_{1}$ and $I_{2}$.  Denote
 by $a_{2}$ the number of blocks in $\pi_{1}$ and $a_{5}$ those in $\pi_{2}$.

 \smallskip

 The blocks in $\pi_{3}$ come in three different forms:

 \smallskip

 \noindent
\textit{Type 1}. The block is of the form $x = \{ \alpha_{1}, \, \beta_{1} \}$ with $\alpha_{1} \in I_{1}$ and
 $\beta_{1} \in I_{2}$.  Let $a_{1}$ be the number of them.  The $n+m$ elements can be placed into these type of blocks in
 \begin{equation*}
 \binom{n}{a_{1}} \binom{m}{a_{1}} a_{1}! \quad \text{ ways}.
 \end{equation*}

 \smallskip

 \noindent
\textit{Type 2}. The form is now $x = \{ \alpha_{1}, \beta_{1}, \beta_{2} \}$ with $\alpha_{1} \in I_{1}$ and
 $\beta_{j} \in I_{2}$, for $j = 1, 2$.  Let $a_{3}$ denote the number of these type of blocks.  These contributed
 \begin{equation*}
 \binom{n}{2a_{3}} \binom{m}{a_{3}}  \frac{(2a_{3})!}{2^{a_{3}}} \quad \text{ to the placement of the } n+m \text{ elements}.
 \end{equation*}

  \smallskip

 \noindent
\textit{Type 3}. The final form is $x = \{ \alpha_{1}, \alpha_{2}, \beta_{1} \}$ with $\alpha_{j} \in I_{1}, \, j=1, \, 2$  and
 $\beta_{1} \in I_{2}$.  Denote by  $a_{4}$ the number of such blocks.  These contribute
 \begin{equation*}
 \binom{n}{a_{4}} \binom{m}{2a_{4}}  \frac{(2a_{4})!}{2^{a_{4}}} \quad \text{ to the count}.
 \end{equation*}

\medskip

Therefore the total number of partitions is given by
\begin{eqnarray*}
\label{mess-1} \\
B_{n+m, \leq 3} =
\sum  \binom{n}{a_{1}, a_2, a_3, 2a_4} \binom{m}{a_{1}, a_5, a_4, 2a_3} a_{1}! \frac{(2a_{3})!}{2^{a_{3}}} \frac{(2a_{4})!}{2^{a_{4}}}
B_{a_{2}, \leq n} B_{a_{5}, \leq m}, \nonumber
\end{eqnarray*}
\noindent
where the sum extends over all indices $0 \leq n_{1}, n_{2}, n_{3}, n_{4}, n_{5}$ such that
\begin{equation*}
a_{1} + a_{2} + a_{3} + 2a_{4} = n  \text{ and } a_{1} + a_{5} + 2a_{3} + a_{4} = m.
\end{equation*}

Introduce the notation $i = n-a_{2}, \, j = m -a_{5}$ and $k = a_{1}$ (so that  $i, \, j, \, k \geq 0$) and
solve for $a_{3}$ and $a_{4}$ from
\begin{eqnarray*}
a_{3} + 2a_{4} &=i-k, \\
2a_{3} + a_{4} &=j-k
\end{eqnarray*}
\noindent
to obtain
\begin{equation}
a_{3} = \frac{2j-i-k}{3} \quad \text{ and }  \quad a_{4} = \frac{2i-j-k}{3}.
\end{equation}
\noindent
The fact that $a_{3}, \, a_{4} \in \mathbb{N}$ is equivalent to $i+j+k \equiv 0 \bmod 3$.  This gives the result.
 \end{proof}

The following theorem is the analog of Theorems \ref{teorecr1}, \ref{req1} and (\ref{genfun3}). The proof is similar, so it is
omitted.

 \begin{theorem}\label{teoas3}
 The restricted factorial sequence  $A_{n, \leq 3}$ is given by
 \begin{equation*}
A_{n, \leq 3}  =  \sum_{j=0}^{\lfloor{ n/3 \rfloor}} \binom{n}{3j} \frac{(3j)!}{3^{j} \, j!} A_{n-3j, \leq 2}.
 \end{equation*}
Moreover, it satisfies the recurrence
 \begin{equation*}
 A_{n, \leq 3} =  A_{n-1, \leq 3} + (n-1) A_{n-2, \leq 3} + (n-1)(n-2)A_{n-3, \leq 3},
 \end{equation*}
 with initial conditions $A_{0, \leq 3} = 1, \, A_{1, \leq 3} = 1, \, A_{2, \leq 3} = 2.$ Its generating function is
  \begin{equation*}
 \sum_{n=0}^{\infty} A_{n, \leq 3} \, \frac{x^{n} }{n!} = \exp{ \left( x + \tfrac{1}{2} x^{2} + \tfrac{1}{3} x^{3} \right)}.
 \end{equation*}
 \end{theorem}

\section{The General Case $B_{n,\leq m}$.}
\label{sec4}

In this section, some recurrences of the restricted Bell numbers are generalized.  A relation between this sequence and the
associated Bell numbers is established. The first statement generalizes   (\ref{eqr1}) and (\ref{rec-1}).

\begin{theorem}
 \label{formular1}
 The restricted Bell numbers $B_{n,\leq m}$ satisfy the recurrence
 \begin{equation}
 B_{n,\leq m}=\sum _{i=0}^{\lfloor \frac{n}{m}\rfloor} \frac{n!}{i!(m!)^i(n-im)!}B_{n-im,\leq m-1}.
 \end{equation}
\end{theorem}
\begin{proof}
Count the set of all partitions of $[n]$ with blocks of size at most $k$ and contain exactly $i$ blocks of size $m$. To do so,
select $m\cdot i$ elements from $n$ without any order. This is done  in
$\binom{n}{\underbrace{m,m,\dots,m}_{\text{ $i$ times}}}=\frac{n!}{m!^i(n-im)!}$ ways. Now  divide by $i!$ to take into account the
 order of the blocks. The $n-im$ remaining elements of $[n]$ are placed in blocks of size $m-1$ or less elements,  counted
  by $B_{n-im,\leq m-1}$. The result follows by summing over $i$.
\end{proof}

A direct argument generalizes Theorems \ref{req1} and (\ref{genfun3}). This result appears in \cite{miksa-1958a}.

 \begin{theorem}
 \label{gfun}
The restricted Bell numbers $B_{n, \leq m}$ satisfy the recurrence
 \begin{equation*}
 B_{n, \leq m} = \sum_{k=0}^{m-1} \binom{n-1}{k} B_{n-k-1, \leq m}.
 \end{equation*}
Moreover, their  exponential generating function is
\begin{equation}
\sum_{n=0}^\infty \frac{B_{n,\leq m}x^n}{n!}=\exp\left(\sum_{i=1}^m \frac{x^i}{i!}\right).
\end{equation}
\end{theorem}

 The next result generalizes  Theorem \ref{teo1a}.

\begin{theorem} \label{teo1ag}
Denote $f(i,j)=2+j+\binom{i-1}{2}$, then
\begin{align*}
B_{n+m,\leq k}=n!m!\sum_{X}\frac{B_{a_1,\leq k}B_{a_2\leq k}}{a_1!a_2!\prod _{i=2}^k\prod _{j=1}^{i-1}{j!^{a_{f(i,j)}}(i-j)!^{a_{f(i,j)}}a_{f(i,j)}!}},
\end{align*}
where $X$ stands for the following set of variables $$X=\{(a_1,a_2,\dots ,a_{1+k+\binom{k-1}{2}}): a_1+\sum _{i=2}^k\sum _{j=1}^{i-1}ja_{f(i,j)}=n \wedge a_2+\sum _{i=2}^k\sum _{j=1}^{i-1}(i-j)a_{f(i,j)}=m \}.$$
\end{theorem}

\begin{proof}
The set of $n + m$ elements, whose partitions are enumerated by $B_{n+m,\leq 3}$, is
split into two disjoint sets $I_1=[n]$ and $I_2=[n+m]\setminus [n]$ of cardinality $n$ and $m$, respectively. Any such
partition $\pi$ can be written uniquely in the form $\pi = \pi _1 \cup \pi _2 \cup \pi _3$, where blocks in $\pi _1$ are subsets of $I_1$, blocks in $\pi _2$ are subsets of $I_2$ and the blocks in $\pi _3$ contain elements of $I_1$ and $I_2$. Denote
 by $a_1$ the number of  elements that are going to be in blocks of $\pi _1$ and by $a_2$ the numbers of elements that
  are going to be in blocks of $\pi _2$. These are counted by $B_{a_1,\leq k}B_{a_2,\leq k}$.

The blocks in $\pi _3$ come in different forms depending in how many elements are in the blocks and how many come from
 $[n]$ and how many from $[n+m]\setminus [n]$. Denote by  $a_{f(i,j)}$  the number of blocks in $\pi _3$ which have $j>0$ elements
 of $[n]$ and $i-j>0$ from $[n+m]\setminus [n]$.  It is required to choose $ja_{f(i,j)}$ elements from $[n]$ and $(i-j)a_{f(i,j)}$
 from $[n+m]\setminus [n]$. The total number of choices for grouping the $a_{f(i,j)}$ blocks is given by
$$\binom{(i-j)a_{f(i,j)}}{\underbrace{i-j,i-j,\dots ,i-j}_{\text{ $a_{f(i,j)}$ times}}}\binom{ja_{f(i,j)}}{\underbrace{j,j,\dots,j}_{\text{$a_{f(i,j)}$ times}}}\frac{1}{a_{f(i,j)}!}=\frac{(ja_{f(i,j)})!((i-j)a_{f(i,j)})!}{j!^{a_{f(i,j)}}(i-j)!^{a_{f(i,j)}}a_{f(i,j)}!}.$$
The multinomial coefficient accounts  for the possible groups of each side and the factorial in the denominator accounts for the
 order of the blocks. Summing over all possible configurations gives the result.
\end{proof}

\subsection{Relations between restricted and associated Bell numbers}

The \textit{associated Stirling numbers of the second kind}  $ {n \brace k}_{\geq m }$ give the number of partitions
 of $n$ elements  into $k$ subsets under the restriction that every blocks contains \textit{at least }$m$ elements. Komatsu et al.
 \cite{komatsu-2015c}  derived the  recurrence
 \begin{align*}
{n+1 \brace k}_{\geq m}  =\sum_{j=m-1}^{n} \binom{n}{j}{n-j \brace k-1}_{\geq m}  = k  {n \brace k}_{\geq m}   +     \binom{n}{m-1} {n-m +1\brace k-1}_{\geq m},
 \end{align*}
     \noindent
with initial conditions $ {0 \brace 0}_{\geq m}   = 1$ and  ${n\brace 0}_{\geq m}  = 0.$ The \textit{associated Bell numbers}
are defined by
  \begin{equation*}
 B_{n, \geq m} = \sum_{k=0}^{n}  {n \brace k}_{\geq m }.
 \end{equation*}
They enumerate partitions of $n$ elements into blocks with the condition that every block has at least than $m$
elements.  For  example, $B_{4, \geq 3} =1 $ with the partition being $ \left\{ \{1 , 2 , 3, 4 \}  \right\}$. Their generating function is
\begin{equation}
\sum_{n=0}^\infty \frac{B_{n,\geq m}x^n}{n!}=\exp\left(\exp(x) - \sum_{i=0}^{m-1} \frac{x^i}{i!}\right).
\end{equation}

In the case $m=2$, $B_{n, \geq 2}$ enumerate partitions of $n$ elements without singleton blocks, it  satisfies  (cf. \cite{bona-2016a})
 \begin{equation}
  \label{idbell2}
 B_{n} = B_{n,\geq 2} + B_{n+1, \geq 2},
 \end{equation}
 \noindent
and its exponential generating function is given by
 \begin{equation}
 \sum_{n=0}^{\infty} B_{n, \geq 2} \, \frac{x^{n} }{n!} = \exp{ \left(\exp(x) - 1- x \right)}.
 \end{equation}
 \noindent
Therefore, the  binomial transform of $B_{n,\geq 2}$ is the Bell sequence $B_n$, i.e.,
\begin{equation}
\sum_{i=0}^n\binom{n}{i}B_{i,\geq 2}=B_n.
\end{equation}
\noindent
The fact that integer sequences and their inverse binomial transform have the same Hankel transform \cite{spivey-2006a}, gives
 the following result.

\begin{theorem}
The Hankel transform of the associated Bell numbers $B_{n,\geq 2}$ is the superfactorials. That is, for  any fixed $n$,
$$\det\begin{bmatrix}
B_{0,\geq 2} & B_{1,\geq 2} & B_{2,\geq 2} & \cdots & B_{n,\geq 2}\\
B_{1,\geq 2} & B_{2,\geq 2} & B_{3,\geq 2} & \cdots & B_{n+1,\geq 2}\\
\vdots & \vdots & \vdots &  & \vdots \\
B_{n,\geq 2} & B_{n+1,\geq 2} & B_{n+2,\geq 2} & \cdots & B_{2n,\geq 2}
\end{bmatrix}=\prod_{i=0}^{n}i!.$$
\end{theorem}

\begin{theorem}
The associated Bell numbers $B_{n,\geq 2}$ and the Bell numbers $B_n$ are related by
$$B_{n,\geq 2}=\sum _{i=0}^n(-1)^i\binom{n}{i}B_{n-i}.$$
\end{theorem}
\begin{proof}
Let $\mathcal{B}_n$ be the set of all partitions of $[n]$, and let $\mathcal{B}_{n,\geq 2}$ be the set of all partitions into blocks of
length at least than  $2$. Denote by $\mathcal{S}_{n,i}$ the set of partitions where $i$ is in a singleton block. Then
\begin{equation}
\mathcal{B}_{n,\geq 2}=\mathcal{B}_{n}\setminus \bigcup _{i\in [n]} \mathcal{S}_{n,i},
\end{equation}
\noindent
and the inclusion-exclusion principle produces
 $$B_{n,\geq 2}=B_n-\sum _{i=1}^n(-1)^{i-1}\sum _{a_1<a_2<\cdots <a_i} \left| \bigcap _{j=1}^i S_{n,a_j} \right|.$$
The identity now follows from $ \left| \displaystyle \bigcap _{j=1}^i S_{n,a_j} \right|=B_{n-i}$.
\end{proof}

The next result gives a reduction for the associated Bell numbers $B_{n, \geq k}$, in the index $k$ counting the minimal number of
elements  in a block.

\begin{theorem}
\label{asbell}
The associated  Bell numbers $B_{n,\geq k}$ satisfy
$$B_{n,\geq k}=B_{n,\geq k-1}-\sum _{i=1}^{\lfloor \frac{n}{k-1}\rfloor}\frac{n!}{(k-1)!^i(n-(k-1)i)!i!}B_{n-(k-1)i,\geq k}.$$
\end{theorem}
\begin{proof}
Denote by $\mathcal{B}_{n,\geq k}$ the set of all partitions with blocks of length at least than $k$. Then
$\mathcal{B}_{n,\geq k}\subseteq \mathcal{B}_{n,\geq k-1}$ and let $A=\mathcal{B}_{n,\geq k-1}\setminus \mathcal{B}_{n,\geq k}$
be the set difference. For $1 \leq k \leq n$, define
$$A_i = \{\pi \in \mathcal{B}_{n,\geq k-1}: \text{ the number of blocks of size $k-1$ is $i$}\},  $$
\noindent
and observe that
$$A:= \bigcup _{i=1}^n A_i =\mathcal{B}_{n,\geq k-1}\setminus \mathcal{B}_{n,\geq k}.$$
The sets $\{A_i\}$ form  a partition of $A$ with
\begin{equation}
|A_i|= \frac{1}{i!} \binom{n}{k-1,k-1,\dots ,k-1}B_{n-i(k-1),\geq k}.
\end{equation}
\noindent
The identity follows from this.
\end{proof}

\begin{theorem}
The associated Bell numbers can be calculated from the Bell numbers and restricted Bell numbers via
\begin{equation}
B_{n,\geq k}=B_n-\sum _{i =1}^n\binom{n}{i}B_{i,\leq k-1}B_{n-i,\geq k}.
\end{equation}
\end{theorem}
\begin{proof}
Recall  that $\mathcal{B}_n$ is the set of partitions of $[n]$. For any such partition, write  $\pi = \{A,B\}$, where
$
A=\{\pi \in \mathcal{B}_n: \text{if }D\in \pi, \text{ then }|D|\geq k \},
$
and $B$ the complement of $A$ in $\mathcal{B}_{n}$. Then $|A|+|B|=B_n$.
Now $|A|=B_{n,\geq k}$ and $B$ can be partitioned in $\{C_i\}_{i\in [n]}$ where $C_i$ contains the partitions such that
 there are exactly $i$ elements of $[n]$ that are in blocks with length less than $k$ and the remaining $n-i$ are in blocks
  with length greater or equal to $k$. Therefore
$$|C_i|=\binom{n}{i}B_{i,\leq k-1}B_{n-i,\geq k},$$ and the result follows by summing over all  partitions of $[n]$.
\end{proof}

 The next result is the analog of Theorem \ref{teo1ag} for the case of the associated Bell numbers.

\begin{theorem}
Denote $f(i,j)=2+j+\binom{i-1}{2}$, then
$$B_{n+m,\geq k}=n!m!\sum_{X}\frac{B_{a_1,\geq k}B_{a_2, \geq k}}{a_1!a_2!\prod _{i=k}^{n+m}\prod _{j=1}^{i-1}{j!^{a_{f(i,j)}}(i-j)!^{a_{f(i,j)}}a_{f(i,j)}!}},$$
where $X$ stands for the following set of variables
\begin{multline*}
X=\left\{ (a_1,a_2,a_{3+\binom{k-1}{2}}, \dots ,a_{2+n+m-1+\binom{n+m-1}{2}}): \right. \\
 \left. a_1+\sum _{i=k}^{n+m}\sum _{j=1}^{i-1}ja_{f(i,j)}=n \wedge a_2+\sum _{i=k}^{n+m}\sum _{j=1}^{i-1}(i-j)a_{f(i,j)}=m \right\}.
 \end{multline*}
\end{theorem}

\section{The General Case $A_{n,\leq m}$.}
\label{sec5}

This section discusses the results presented in the previous section corresponding to the class $A_{n, \leq m}$.

\smallskip

The first statement  generalizes  Theorem \ref{teoas3} and is the analog of  Theorem \ref{formular1}. The proof is similar to the
one given above. Details are omitted.

 \begin{theorem}
 The restricted factorial numbers  $A_{n, \leq m}$ are given by
 \begin{equation*}
A_{n, \leq m}  =\sum _{i=0}^{\lfloor \frac{n}{m}\rfloor} \frac{n!}{m^ii!(n-im)!}A_{n-im,\leq m-1}.
 \end{equation*}
\end{theorem}

The next statement is found in \cite{mezo-2014a}.

\begin{theorem}
\label{teo52}
 The restricted factorial sequence $A_{n, \leq m}$ satisfies the recurrence
 \begin{equation*}
 A_{n, \leq m} =  \sum_{j=0}^{m-1} \frac{(n-1)!}{(n-1-j)!} A_{n-1-j,\leq m},
 \end{equation*}
 with initial conditions $A_{0, \leq m} = 1\ \text{and}  \ A_{1, \leq m} = 1.$
 Its generating function is
  \begin{equation*}
 \sum_{n=0}^{\infty} A_{n, \leq m} \, \frac{x^{n} }{n!} = \exp{ \left( x + \tfrac{1}{2} x^{2} + \tfrac{1}{3} x^{3} +\cdots + \tfrac{1}{m}x^m\right)}.
 \end{equation*}
 \end{theorem}

 The next reduction formula gives $A_{n+m, \leq k}$ in terms of lower value of the first index.

\begin{theorem}\label{spiveyres}
Denote $f(i,j)=2+j+\binom{i-1}{2}$, then
 $$A_{n+m,\leq k}=n!m!\sum_{X}\frac{A_{a_1,\leq k}A_{a_2\leq k}}{a_1!a_2!}\prod _{i=2}^k\prod _{j=1}^{i-1}\binom{i}{j}^{a_{f(i,j)}}\frac{1}{i^{a_{f(i,j)}} \cdot a_{f(i,j)}!},$$
where $X$ stands for the following set of variables $$X=\{(a_1,a_2,\dots ,a_{1+k+\binom{k-1}{2}}): a_1+\sum _{i=2}^k\sum _{j=1}^{i-1}ja_{f(i,j)}=n \wedge a_2+\sum _{i=2}^k\sum _{j=1}^{i-1}(i-j)a_{f(i,j)}=m \}.$$
 \end{theorem}

\begin{example}
The special case $k=3$ gives
$$A_{n+m, \leq 3} = \sum _{i=0}^n\sum _{j=0}^m\sum _{\overset{l=0}{{l\equiv -n-m+i+j\bmod 3}} }^{\min \{n-i,m-j\}}\frac{n!m!A_{i,\leq 3}A_{j,\leq 3}}{i!j!l! \left( \frac{2m-n+i-2j-l}{3} \right)! \, \left(\frac{2n-m-2i+j-l}{3}\right)!}.$$
\end{example}

\subsection{The Associated Factorial Numbers  $A_{n, \geq m} $}

This section presents analogous results for sequence built from the \textit{associated  Stirling numbers of the first kind}
  $ {n \brack k}_{\geq m }$. These numbers satisfy the following recurrence \cite{komatsu-2015b}
 \begin{align*}
{n+1 \brack k}_{\geq m}  =\sum_{j=m-1}^{n} \frac{n!}{(n-j)!}
 {n-j \brack k-1}_{\geq m} = n {n \brack k}_{\geq m}  +   \frac{n!}{(n-m+1)!} {n-m + 1 \brack k-1}_{\geq m},
 \end{align*}
     \noindent
with the initial conditions
$
 {0 \brack 0}_{\geq m}   = 1
      \text{ and }
      {n\brack 0}_{\geq m}  = 0.
   $
The \textit{associated   factorial numbers} defined by
 \begin{equation*}
 A_{n, \geq m} = \sum_{k=0}^{n}  {n \brack k}_{\geq m },
 \end{equation*}
 enumerate all permutations  of $n$ elements into cycles with the condition that every cycle has at least than $m$ items. Its
 generating function is given by \cite{wilf-1990a}
  \begin{equation*}
 \sum_{n=0}^{\infty} A_{n, \geq m} \, \frac{x^{n} }{n!} = \exp\left(\sum_{n=m}^{\infty}\frac{x^n}{n}\right)=\exp\left(\log\frac{1}{1-x}-\sum_{n=1}^{m-1}\frac{x^n}{n}\right).
 \end{equation*}
 In particular, if $m=2$  we obtain the number of permutations of $n$ elements with no fixed points, the classical
 derangements numbers. This sequence satisfies that (cf. \cite{bona-2012a})
 \begin{align}
 A_{n, \geq 2} &= nA_{n-1, \geq 2} + (-1)^{n}, \quad  n\geq 1, \label{bonab}\\
 &=(n-1)(A_{n-1,\geq2}+A_{n-2,\geq 2}). \label{bonab2}
 \end{align}
 \noindent
Radoux \cite{radoux-1991a} has shown that the  Hankel transform of the associated factorial numbers $A_{n,\geq 2}$ is given by
$\begin{displaystyle}\prod_{i=1}^{n}i!^2.\end{displaystyle}$

\smallskip

The following theorem is the analog  of Theorem \ref{asbell}, with a similar proof. The details are omitted.

\begin{theorem}
For $n, \, k  \in \mathbb{N}$ with  $k > 1$, the associated  factorial numbers $A_{n,\geq k}$ satisfy
$$A_{n,\geq k}=A_{n,\geq k-1}-\sum _{i=1}^{\lfloor \frac{n}{k-1}\rfloor}\frac{n!}{(k-1)^i(n-(k-1)i)!i!}A_{n-(k-1)i,\geq k}.$$
\end{theorem}

The following result corresponds to  Theorem \ref{spiveyres}.

\begin{theorem}
Denote $f(i,j)=2+j+\binom{i-1}{2}$, then
 $$A_{n+m,\geq k}=n!m!\sum_{X}\frac{A_{a_1,\geq k}A_{a_2\geq k}}{a_1!a_2!}\prod _{i=k}^{n+m}\prod _{j=1}^{i-1}\binom{i}{j}^{a_{f(i,j)}}\frac{1}{i^{a_{f(i,j)}}\cdot a_{f(i,j)}!},$$
where $X$ stands for the following set of variables $$X=\{(a_1,a_2,\dots ,a_{1+k+\binom{k-1}{2}}): a_1+\sum _{i=k}^{n+m}\sum _{j=1}^{i-1}ja_{f(i,j)}=n \wedge a_2+\sum _{i=k}^{n+m}\sum _{j=1}^{i-1}(i-j)a_{f(i,j)}=m \}.$$
 \end{theorem}

The next statement generalizes \eqref{bonab2}.

\begin{theorem}\label{genAigner}
The associated  factorial numbers $A_{n,\geq k}$ satisfy
$$A_{n,\geq k}=(n-1)A_{n-1,\geq k}+ (n-1)^{\underline{k-1}}A_{n-k,\geq k}, \quad n\geq 1,$$
where $n^{\underline{k}}:=n(n-1)\cdots (n-(k-1)) = \begin{displaystyle} \frac{n!}{(n-k)!} \end{displaystyle}$ and $n^{\underline{0}}=1$.
\end{theorem}
\begin{proof}
Denote by $\mathcal{A}_{n,\geq k}$ the permutations $\sigma$  on $n$ elements such that the length of every cycle in
$\sigma$  is not less than $k$ (i.e.,
$\mathcal{A}_{n,\geq k}=\{\sigma \in \mathcal{S}_n:| \langle i \rangle|\geq k\}$. Here
$\langle i \rangle$ denotes the cycle of $i\in [n]$ as a set).  For $\sigma \in \mathcal{A}_{n,\geq k}$, there are two
cases for $n\in[n]$:
\begin{itemize}
\item \textit{Case 1}: here
$|\langle n \rangle |=k$. It is required to construct a cycle of length $k$ containing $n$. In order to do this, one
must choose $k-1$ numbers from $[n-1]$ and place them in the same cycle. This can be done in $\binom{n-1}{k-1}$ ways and the
total number of
possible cycles is $\binom{n-1}{k-1}(k-1)!=(n-1)^{\underline{k-1}}$. The other cycles are counted by $A_{n-k,\geq k}$, for a total
of $(n-1)^{\underline{k-1}}A_{n-k,\geq k}$.

\item \textit{Case 2}: now $| \langle n \rangle|>k$. Then one needs to place $n$ in any
cycle of a permutation $\sigma ' \in \mathcal{A}_{n-1,\geq k}$. There are $(n-1)A_{n-1,\geq k}$ ways to do it.
\end{itemize}

\noindent
The identity follows from this discussion.

\end{proof}

\section{Log-Convex  and Log-Concavity Properties}
\label{sec-logconvex}

A sequence $(a_n)_{n\geq0}$ of nonnegative real numbers is called \textit{log-concave}
if $a_na_{n+2}\leq a_{n+1}^2$, for all $n\geq 0$. It is called \textit{log-convex} if $a_na_{n+2}\geq a_{n+1}^2$ for all $n\geq 0$. There
is a large collection of results on log-concavity and log-convexity and its relation to combinatorial sequences. Some of these appear in
\cite{brenti-1989a}, \cite{mcnamara-2010a}, \cite{medinal-2016a}, \cite{sagan-1998a} and \cite{wilf-1990a}. The Bell sequence is
  log-convex \cite{asai-2000a} and it is not difficult to verify that the same is true for restricted Bell numbers
and restricted factorial numbers.

A sequence $(a_n)_{n\in\N}$ has no internal zeros if there do not exist integers $0\leq i<j<k$ such that $a_i\neq 0, a_j=0, a_k\neq0$.

\begin{theorem}[Bender-Canfield Theorem, \cite{bender-1996a}]\label{BCT}
Let $\{ 1, w_1, w_2,\dots \}$ be a log-concave sequence of  nonnegative real numbers
with no internal zeros. Define the sequence $(a_n)_{n\geq0}$  by
$$\sum_{n=0}^\infty \frac{a_n}{n!}x^n=\exp\left(\sum_{j=1}^\infty \frac{w_i}{i}x^j\right). $$
Then the sequence $(a_n)_{n\geq0}$ is log-convex and the sequence $(a_n/n!)_{n\geq0}$  is log-convave.
\end{theorem}

This result is now used to verify the log-convexity of $B_{n, \leq m}$.

\begin{theorem}
The restricted Bell sequence $(B_{n,\leq m})_{n\geq 0}$ is log-convex and the sequence \linebreak $(B_{n,\leq m}/n!)_{n\geq 0}$ is log-concave.
\end{theorem}
\begin{proof}
The result follows from Theorems \ref{gfun} and \ref{BCT} and the log-concavity of the sequence
$$w_i=\begin{cases}
\frac{1}{(i-1)!},& \ \text{if} \ 1\leq i \leq m;\\
0,& i>m. \end{cases}$$
\end{proof}

The next statement is similar.

\begin{theorem}
The restricted factorial sequence $(A_{n,\leq m})_{n\geq 0}$ is log-convex and the sequence $(A_{n,\leq m}/n!)_{n\geq 0}$ is log-concave.
\end{theorem}
\begin{proof}
Now use Theorems \ref{teo52} and \ref{BCT} and the sequence
$$w_i=\begin{cases}
1,& \ \text{if} \ 1\leq i \leq m;\\
0,& i>m. \end{cases}$$
to produce the result.
\end{proof}

\subsection{Open questions}
Some conjectured statements are collected here.
The restricted Bell polynomials are defined by $$B_{n,\leq m}(x):=\sum_{k=0}^n{n \brace k}_{\leq m} x^k.$$

The recurrence (\ref{recur-1}), produces
 $$B_{n+1,\leq m}(x)=xB_{n,\leq m}(x)+xB'_{n,\leq m}(x)-\binom{n}{m}xB_{n-m,\leq m}(x). \quad $$
This can be verified directly:
 \begin{align*}
 B_{n+1,\leq m}(x)&=x\sum_{k=0}^{n}  k{n \brace k}_{\leq m}x^{k-1}  + x\sum_{k=0}^{n} {n \brace k}_{\leq m}x^{k}  -    \binom{n}{m} x\sum_{k=0}^{n-m} {n-m \brace k}_{\leq m} x^{k} \\
  &= xB'_{n,\leq m}(x) + xB_{n,\leq m}(x)-\binom{n}{m}xB_{n-m,\leq m}(x).
 \end{align*}
 \noindent
 The authors have tried, without success, to establish  the next two statements:

\begin{conjecture}
\label{conj1}
The roots of the polynomial $B_{n,\leq m}(x)$ are real and non-positive if $m \neq 3, 4$.
\end{conjecture}

Recall that a finite sequence $\{ a_{j}, \, 0 \leq j \leq n \}$  of non-negative real numbers is called \textit{unimodal} if there is an index
$j^{*}$ such that $a_{j- 1} \leq a_{j}$ for $1 \leq j \leq j^{*}$ and $a_{j-1} \geq a_{j}$ for $j^{*} +1 \leq j \leq n$. An elementary argument
shows that a log-concave sequence must be unimodal. The unimodality of the restricted Stirling numbers
 $\left({n \brace k}_{\leq 2} \right)_{k\geq 0}$ was proved by
Choi and Smith in \cite{choijy-2003a}. Moreover, Han and Seo
 \cite{hanh-2004a} gave a combinatorial proof of the log-concavity of this sequence. The log-concavity of the associated
 Stirling numbers of the first kind was studied by Brenti in \cite{brenti-1993a}.

 \begin{conjecture}
 \label{conj2}
The sequence of restricted Stirling numbers $\left({n \brace k}_{\leq m} \right)_{k\geq 0}$ is log-concave.
\end{conjecture}

One of the main sources of log-concave sequences comes from the following fact: if $P(x)$ is a polynomial all of whose zeros
are real and negative, the its coefficient sequence is log-concave. (See \cite[Theorem 4.5.2]{wilf-1990a} for a proof). Therefore
the first conjecture  implies the second one.

\section{Some Arithmetical Properties}
\label{sec-arith}

Given a prime $p$, the $p$-adic valuation of $x \in \mathbb{N}$, denoted by $\nu_{p}(x)$, is the highest power of $p$ that divides
 $x$.  For a given sequence of positive integers $(a_{n})_{n\geq 0}$ a description of the sequence of valuations
 $\nu_{p}(a_{n})$ often presents interesting questions. The classical formula of Legendre for factorials
\begin{equation*}
\nu_{p}(n!) = \sum_{k=1}^{\infty} \left\lfloor \frac{n}{p^{r}} \right\rfloor
\end{equation*}
\noindent
is one of the earliest such descriptions. This may also be expressed in closed-form as
\begin{equation*}
\nu_{p}(n!) = \frac{n - s_{p}(n)}{p-1},
\end{equation*}
\noindent
where $s_{p}(n)$ is the sum of the digits of $n$ in its expansion in base $p$.  The reader will find in \cite{amdeberhan-2008a,
amdeberhan-2008b,byrnes-2015a,cohn-1999a,kamano-2011a,moll-2010a,straub-2009a,sunx-2009a} a
selection of results on this topic.

\smallskip

The 2-adic valuation of the Bell numbers has been described in \cite{amdeberhan-2013f}.

\begin{theorem}\label{2adbell}
The 2-adic valuation of the Bell numbers satisfy $\nu_2(B_n)=0$ if  $n \equiv 0, 1 \bmod 3$. In the missing
case, $n\equiv 2 \bmod 3$,  $\nu_2(B_{3n+2})$ is a periodic
 sequence of period 4. The repeating values are $\{1, 2, 2, 1\}$.
 \end{theorem}

The  $2$-adic valuation of the restricted Bell sequence $B_{n, \leq 2}$  was described in \cite{amdeberhan-2015b}.

\begin{theorem}
\label{val-bell}
The 2-adic valuation of the restricted Bell numbers $B_{n, \leq 2}$ satisfy
\begin{equation}
\nu_{2}(B_{n, \leq 2}) =  \left\lfloor \frac{n}{2} \right\rfloor -2 \left\lfloor \frac{n}{4} \right\rfloor +
\left\lfloor \frac{n+1}{4} \right\rfloor =
\begin{cases}
k, & \quad \text{ if } n = 4k; \nonumber \\
k, & \quad \text{ if } n = 4k+1;  \nonumber \\
k+1, & \quad \text{ if } n = 4k+2;  \nonumber \\
k+2, & \quad \text{ if } n = 4k+3. \nonumber
\end{cases}
\end{equation}
\end{theorem}

This section discusses  the 2-adic valuation of the numbers $B_{n, \geq 2}$ and $A_{n, \geq 2}$. Figure \ref{figurea}  shows
 the first 100 values.

\begin{figure}[H]
\begin{center}
\includegraphics[scale=0.8]{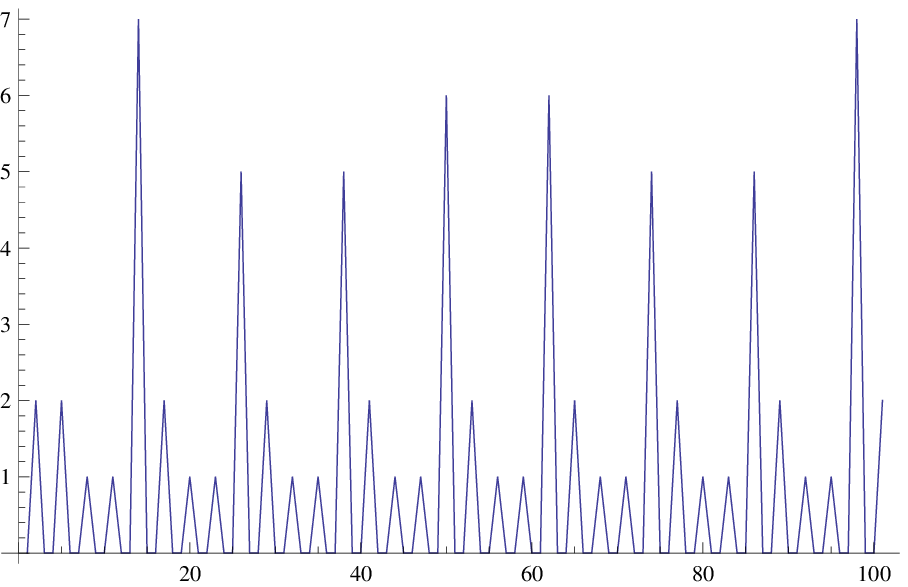}
\includegraphics[scale=0.8]{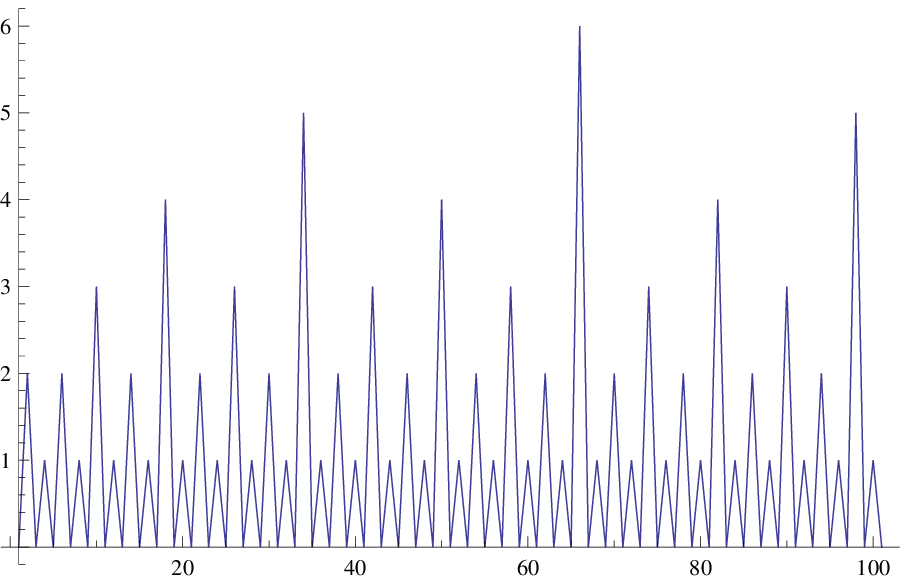}
\caption{The $2$-adic valuation of $B_{n,\geq 2}$ and $A_{n, \geq 2}$.}
\label{figurea}
\end{center}
\end{figure}

\begin{theorem}
The 2-adic valuation of the associated Bell numbers $B_{n, \geq 2}$ is given by
\begin{equation}
\nu_2(B_{n, \geq 2})=0 \text{  if }n\equiv 0, \, 2 \bmod 3.
\end{equation}
\noindent
For $n \equiv 1 \bmod 3$, the valuation satisfies $\nu_{2}(B_{n, \geq 2}) \geq 1$.
\end{theorem}
\begin{proof}
The proof is by induction on $n$. Divide the analysis  into three cases according to the residue of $n$ modulo 3.  If
$n=3k$ then  \eqref{idbell2} gives $B_{3k-1}=B_{3k-1,\geq 2} + B_{3k, \geq 2}$. Theorem
 \ref{2adbell} shows that $B_{3k-1}$  is even and by the induction hypothesis  $B_{3k-1,\geq 2}$ is odd. Thus
 $B_{3k, \geq 2}$ is
 odd, so that $\nu_2(B_{3k, \geq 2})=0$. The proof is analogous for the case $3k+2$.  The case $n \equiv 1 \bmod 3$ follows
 from the identity \eqref{idbell2} in the form $B_{3k} = B_{3k,\geq 2} + B_{3k+1, \geq 2}$
 and the fact that $B_{3k}$ is odd (by Theorem \ref{val-bell}) and so is $B_{3k, \geq 2}$ by the previous analysis.
\end{proof}

A partial description of the valuations of $B_{n, \geq 2}$ for $n \equiv 1 \bmod 3$ is given in the next conjecture.

\begin{conjecture}
The sequence of valuations  $\nu_{2}(B_{3k+1, \geq 2})$ satisfies the following pattern:  \begin{equation}
\nu_{2}(B_{n, \geq 2}) = \begin{cases}
2, & \text{ if } n \equiv 4  \,\,\,\,\,\,\,\,\,\,\,\, \bmod 12; \\
1, & \text{ if } n \equiv 7, \, 10 \,\,  \bmod 12.
\end{cases}
\end{equation}
\noindent
The remaining case $n \equiv 1 \bmod 12$, considered modulo $24$, obeys the rule
\begin{equation}
\nu_{2}(B_{24n+1, \geq 2}) = 5 + \nu_{2}(n), \quad \text{ for } n \geq 1,
\end{equation}
\noindent
with the case $n \equiv 13 \bmod 24$ remaining to be determined. Continuing this process yields the conjecture
\begin{equation}
\nu_{2}(B_{48n+37, \geq 2}) = 5  \text{ and } \nu_{2}(B_{96n+61, \geq 2}) = 6.
\end{equation}
\noindent
The details of this analysis will appear elsewhere.
\end{conjecture}

A closed-form for the  valuation $\nu_{2}(A_{n, \geq 2})$ is simpler to obtain.

\begin{theorem}
The 2-adic valuation of the associated factorial numbers $A_{n, \geq 2}$ is given by \begin{align*}
\nu_2(A_{n, \geq 2})=
\begin{cases}
0, & \text{if} \ n=2k \,\,\,\,\,\,\,\,\,\, \text{ and } k \geq 0;\\
\nu_2(k)+1, & \text{if} \ n=2k+1 \text{ and } k \geq 1.
\end{cases}
\end{align*}
\end{theorem}
\begin{proof}
If $n$ is even, then  \eqref{bonab} shows that $A_{n, \geq 2}$ is odd, so that  $\nu_2(A_{n, \geq 2})=0$.  If $n$ is odd
then  \eqref{bonab2} gives $\nu_2(A_{2k+1, \geq 2})=\nu_2(2k)=\nu_2(k)+1$.
\end{proof}


\subsection{Some additional patterns}
In this subsection we show some additional examples of the $p$-adic valuation of the restricted and associated Bell and factorial sequences.\\

Theorems \ref{teoas3} and \ref{genAigner} are now used to produce explicit formulas for
 the 2-adic valuation of the restricted and associated factorial numbers for $m=3$.

\begin{theorem}
The 2-adic valuation of the restricted factorial numbers $A_{n, \leq 3}$, for $n \geq 1$,  is given by \begin{align*}
\nu_2(A_{n, \leq 3})=
\begin{cases}
k, & \text{if} \ n=4k;\\
k, & \text{if} \ n=4k+1;\\
k+1, & \text{if} \ n=4k+2;\\
k+1, & \text{if} \ n=4k+3.
\end{cases}
\end{align*}
\end{theorem}
\begin{proof}
The proof is by induction on $n$. It is divided into four cases according to the residue of $n$ modulo 4.  The
symbols $O_i$ denote an odd number. If $n=4k$ then Theorem \ref{teoas3} and the induction hypothesis give
 \begin{align*}
 A_{4k, \leq 3}&=A_{4k-1, \leq 3}+(4k-1)A_{4k-2, \leq 3}+(4k-1)(4k-2)A_{4k-3, \leq 3}\\
 &=2^kO_1+(4k-1)2^kO_2+(4k-1)(4k-2)2^{k-1}O_3\\
 &=2^k(O_1+(4k-1)O_2+(4k-1)(2k-1)O_3)\\
 &=2^kO_4.
\end{align*}
Therefore $\nu_2(A_{4k, \leq 3})=k$.  The remaining cases are analyzed in a similar manner.
\end{proof}

\begin{theorem}
The 2-adic valuation of the associated factorial numbers $A_{n, \geq 3}$, for $n \geq 1$, is given by \begin{align*}
\nu_2(A_{n, \geq 3})=
\begin{cases}
k, & \text{if} \ n=4k;\\
\nu_2(k)+k+2, & \text{if} \ n=4k+1;\\
\nu_2(k)+k+4, & \text{if} \ n=4k+2;\\
k+1, & \text{if} \ n=4k+3.
\end{cases}
\end{align*}
\end{theorem}
\begin{proof}
The proof is as in the previous theorem. If $n=4k$ then  Theorem \ref{genAigner} and the
 induction hypothesis give
 \begin{align*}
 A_{4k, \geq 3}&=(4k-1)A_{4k-1, \geq 3}+(4k-1)(4k-2)A_{4k-3, \geq 3}\\
 &=2^kO_1+(4k-1)(4k-2)2^{\nu_2(k-1)+k+1}O_2\\
 &=2^k(O_1+(4k-1)(2k-1)2^{\nu_2(k-1)+2}O_2)\\
 &=2^kO_3.
\end{align*}
Therefore $\nu_2(A_{4k, \geq 3})=k$.

\smallskip

If $n=4k+1$ then  Theorem \ref{genAigner} and the induction hypothesis now give
 \begin{align*}
 A_{4k+1, \geq 3}&=(4k)A_{4k, \geq 3}+(4k)(4k-1)A_{4k-2, \geq 3}\\
 &=2^{k+2}kO_1+(4k)(4k-1)2^{\nu_2(k-1)+k+3}\\
 &=2^{k+2}kO_1+k2^{\nu_2(k-1)+k+5}O_3)\\
 &=2^{k+2}k(O_1+2^{\nu_2(k-1)+3}O_3)\\
 &=2^{k+2}kO_4.
\end{align*}
Therefore $\nu_2(A_{4k+1, \geq 3})=\nu_2(k)+k+2$.   The remaining cases are analyzed in a similar manner.
\end{proof}

Divisibility properties of the sequences $B_{n, \leq 2}$ and $B_{n, \leq 3}$ by the prime $p=3$ turn out to be much simpler: $3$ does not divide any element of this sequence. The proof is based on the recurrences \eqref{recBn2}  and \eqref{rec-stir1}.

\medskip

\begin{theorem}
The sequence of residues $B_{n, \leq 2}$ modulo $3$ is a periodic sequence of period $3$, with fundamental period $\{ 1, \, 1, \, 2\}$.
\end{theorem}
\begin{proof}
Assume $n \equiv 0 \bmod 3$ and write $n = 3k$. Then \eqref{recBn2} gives
\begin{eqnarray*}
B_{3k, \leq 2} & = &  B_{3k-1, \leq 2} + (3k-1)B_{3k-2, \leq 2} \\
         & \equiv & 2 - 1 = 1 \bmod 3,
   \end{eqnarray*}
   \noindent
and $B_{n, \leq 2} \equiv 1 \bmod 3$. The remaining two cases for $n$ modulo $3$ are treated in the same form.
\end{proof}

\begin{theorem}
The sequence of residues $B_{n, \leq 3}$ modulo $3$ is a periodic sequence of period $6$, with fundamental period
$\{ 1, \, 1, \, 2, \, 2, \, 2, \, 1 \}$.
\end{theorem}
\begin{proof}
Assume $n \equiv 0 \bmod 6$ and write $n = 6k$. Then \eqref{rec-stir1} gives
\begin{eqnarray*}
B_{6k, \leq 3} & = &  B_{6k-1, \leq 3} + (6k-1) B_{6k-2, \leq 3} + (3k-1)(6k-1)B_{6k-3, \leq 3}  \\
         & \equiv & 1 - 2 + 2 = 1 \bmod 3,
   \end{eqnarray*}
   \noindent
showing that  $B_{n, \leq 3} \equiv 1 \bmod 3$. The remaining five cases for $n$ modulo $6$ are treated in the same form.
\end{proof}

\begin{corollary}
The restricted Bell numbers $B_{n,\leq 2}$ and $B_{n,\leq 3}$ are not divisible by $3$.
\end{corollary}

\medskip

Using this type of analysis it is possible to prove the following results:

\begin{itemize}
\item The $5$-adic valuation of the sequence $B_{n, \leq 3}$ is given by
\begin{equation*}
\nu_{5}(B_{n, \leq 3}) = \begin{cases}
1, & \quad \text{ if } n \equiv 3 \bmod 5; \\
0, & \quad \text{ if } n \not \equiv 3 \bmod 5.
\end{cases}
\end{equation*}
\item
The $7$-adic valuation of the sequence $B_{n, \leq 3}$ satisfies $\nu_{7}(B_{n, \leq 3}) = 0$ if $n \not \equiv 4 \bmod 7$.
\item
The sequence of residues $B_{n, \leq 5}$ modulo $7$ is a periodic sequence of period $7$, with fundamental period $\{1, \ 1,\  2,\  5,\  1,\  3,\  6 \}$.
\item
The 3-adic valuation of the associated factorial numbers $A_{n, \geq 3}$ satisfy $\nu_3(A_{n, \geq 3})=0$ if  $n \equiv 0 \bmod 3$. For $n=3k+1$,  the valuation is given by  $\nu_3(A_{n, \geq 3})=\nu_3(A_{n+1, \geq 3})=\nu_3(n-1)$. This covers all cases.

\item
The sequence of residues $A_{n, \leq 5}$ modulo $7$ is a periodic sequence of period $7$, with fundamental period $\{1, \ 1, \ 2,  \ 6, \ 3, \ 1, \ 5\}$.
\end{itemize}

Many other results of this type can be discovered experimentally.  A discussion of a general theory is in preparation.

\bigskip

\noindent
{\bf Acknowledgements}. The first author acknowledges the
partial support of NSF-DMS 1112656. The last author is a graduate student at Tulane University.  The first author
thanks an invitation from the Department of Mathematics of Universidad Sergio Arboleda, Bogot\'{a}, Colombia, where the
presented work was initiated.

\end{document}